\documentclass[12pt]{amsart}

\usepackage[utf8]{inputenc}
\usepackage{tikz, calc}
\usepackage[justification=centering]{caption}
\usepackage{subcaption}
\usepackage{amssymb, amsthm, fullpage,hyperref}
\usepackage{graphicx}
\usepackage{soul}

\newtheorem{definition}{Definition}
\newtheorem{lemma}{Lemma}
\newtheorem{theorem}{Theorem}
\newtheorem{proposition}{Proposition}
\newtheorem{conjecture}{Conjecture}

\newcounter{x}
\newcounter{y}
\newcounter{x2}
\newcounter{y2}

\newcommand{\vertex}[4][black,]{
    \filldraw[#1] (#2, #3) circle (3pt) node[anchor=west]{#4};
}
\newcommand{\tikzbox}[4][thick,-]{
    \setcounter{x}{#2}
    \setcounter{y}{#3}
    \setcounter{x2}{#2 + #4}
    \setcounter{y2}{#3 + #4}
    
    \draw[#1] (\arabic{x},\arabic{y}) -- (\arabic{x2},\arabic{y2});
    
    \setcounter{y}{#3 + #4 + #4}
    \draw[#1] (\arabic{x2},\arabic{y2}) -- (\arabic{x},\arabic{y});
    
    \setcounter{x2}{#2 - #4}
    \draw[#1] (\arabic{x},\arabic{y}) -- (\arabic{x2},\arabic{y2});
    
    \setcounter{y}{#3}
    \draw[#1] (\arabic{x2},\arabic{y2}) -- (\arabic{x},\arabic{y});
}
\newcommand{\grid}[3][step=1cm, gray, very thin,]{
    \draw[#1] (-0.9,-0.9) grid (#2 + 0.9, #3 + 0.9);
}
\newcommand{\tower}[5][thick, black]{
    \tikzbox{#2}{#3-#4}{#4};
    \vertex{#2}{#3}{#5};
}
\newcommand{\halfsquaregrid}[3][step=1cm]{
    \clip (-0.9, -0.9) rectangle (#2 + 0.9, #3 + 0.9);
    \grid[gray, thin]{#2}{#3};
    \draw[rotate=45, shift={(-0.707*#2-2*0.707, -0.707*#2-2*0.707)}, dotted, thin, gray, scale=0.707]  (0,0) grid (2*#2+2*#3+2, 2*#2+2*#3+2);
}
\newcommand{\horizhalfsquare}[5][black, fill=white, thick]{
    \draw[#1, rotate around={45: (#2, #3)}] (#2, #3) rectangle (#2 + 0.707, #3 + 0.707);
}
\newcommand{\verthalfsquare}[5][black, fill=white, thick]{
    \draw[#1, rotate around={-45: (#2, #3)}] (#2, #3) rectangle (#2 + 0.707, #3 + 0.707);
}
\newcommand{\li}[5][thick]{
    \draw [#1,thick] (#2, #3) -- (#4, #5);
}
\newcommand{\holehsvert}[3][fill=lightgray, thick]{
    \fill [#1, rotate around={-45: (#2, #3)}] (#2, #3) rectangle (#2 + 0.707, #3 + 0.707);
}
\newcommand{\holehshoriz}[3][fill=lightgray, thick]{
    \fill [#1, rotate around={45: (#2, #3)}] (#2, #3) rectangle (#2 + 0.707, #3 + 0.707);
}
\newcommand{\hole}[5][fill=lightgray, thick]{
    \fill [#1, rotate around={-45: (#2, #3)}] (#2, #3) rectangle (#2 + 0.707*#5, #3 + 0.707*#4);
}
\newcommand{\redoutline}[5][red, thick]{
    \draw [#1, rotate around={-45: (#2, #3)}] (#2+0.03, #3-0.03) rectangle (#2+0.707*#5-0.03, #3+0.707*#4-0.03);
}
\newcommand{\redsquare}[2]{
    \redoutline{#1}{#2}{4}{4};
}
\newcommand{\smredsquare}[2]{
    \redoutline[red,thick]{#1+0.1}{#2}{4-0.14}{4-0.14};
}

\newcommand{\ohalfsquare}[2]{
    \fill [orange, rotate around={-45: (#1, #2)}] (#1, #2) rectangle (#1 + 0.707, #2 + 0.707);
}
\newcommand{\yhalfsquare}[2]{
    \fill [yellow, rotate around={-45: (#1, #2)}] (#1, #2) rectangle (#1 + 0.707, #2 + 0.707);
}
\newcommand{\figeleven}[1]{
    \hole{2+#1}{4}{2}{1}
    \redoutline{0+#1}{2}{4}{2};
    \redoutline{3+#1}{5}{4}{2};
    \li[]{2+#1}{4}{2.5+#1}{3.5}
    \li[]{3+#1}{5}{3.5+#1}{4.5}
    \li[]{2.5+#1}{3.5}{3.5+#1}{4.5}
    \vertex[]{2+#1}{4}{}
    \vertex[]{3+#1}{5}{}
    \vertex[]{3+#1}{4}{$v$}
}
\newcommand{\figendbox}[1]{
        \hole[lightgray]{2+#1}{4}{1}{1}
        \li[]{2+#1}{4}{2.5+#1}{3.5}
        \li[]{2.5+#1}{3.5}{3+#1}{4}
        \li[]{2+#1}{4}{2.5+#1}{4.5}
        \vertex[]{2+#1}{4}{}
        \vertex[]{3+#1}{4}{}
        \node at (2+#1,3.5) {$e$};
        \node at (3+#1,3.5) {$f$};
}
\newcommand{\figendboxthree}[1]{
        \redsquare{-1+#1}{3}
        \redsquare{2+#1}{3}
        \hole[lightgray]{2+#1}{4}{3}{1}
        \li[]{2+#1}{4}{2.5+#1}{3.5}
        \li[]{2.5+#1}{3.5}{4+#1}{5}
        \li[]{2+#1}{4}{3.5+#1}{5.5}
        \vertex[]{2+#1}{4}{}
        \vertex[]{3+#1}{4}{$v$}
        \vertex[]{3+#1}{5}{}
        \vertex[]{4+#1}{5}{}
}
\newcommand{\rbracket}[3]{
    \li[blue, ultra thick]{#1}{#2}{#1+1.5}{#2-0.5};
    \li[blue, ultra thick]{#1+1.5}{#2-0.5}{#1+1}{#2+1};
    \node[blue] at (#1+1.75,#2-0.75) {#3};
}
\newcommand{\lbracket}[3]{
    \li[blue, ultra thick]{#1}{#2}{#1-0.5}{#2+1.5};
    \li[blue, ultra thick]{#1-0.5}{#2+1.5}{#1+1}{#2+1};
    \node[blue] at (#1-0.75,#2+1.5) {#3};
}

\begin{document}

\title{Optimal $(t,r)$ Broadcasts On the Infinite Grid}

\author{Benjamin F. Drews}
\address{Department of Mathematics and Statistics, Williams College, United States}
\email{bfd2@williams.edu}
\thanks{}

\author{Pamela E. Harris}
\address{Department of Mathematics and Statistics, Williams College, United States}
\email{peh2@williams.edu}
\thanks{P.\,E. Harris was supported by NSF award DMS-1620202.}

\author{Timothy W. Randolph}
\address{Department of Mathematics and Statistics, Williams College, United States}
\email{twr2@williams.edu}
\thanks{}

\date{\today}

\maketitle
\begin{abstract}
Let $G=(V,E)$ be a graph and $t,r$ be positive integers. The \emph{signal} that a vertex $v$ receives from a tower of signal strength $t$ located at vertex $T$ is defined as $sig(v,T)=max(t-dist(v,T),0),$ where $dist(v,T)$ denotes the distance between the vertices $v$ and $T$. In 2015 Blessing, Insko, Johnson, and Mauretour defined a \emph{$(t,r)$ broadcast dominating set}, or simply a \emph{$(t,r)$ broadcast}, on $G$ as a set $\mathbb{T}\subseteq V$ such that the sum of all signal received at each vertex $v \in V$ is at least $r$. We say that $\mathbb{T}$ is \emph{optimal} if $|\mathbb{T}|$ is minimal among all such sets $\mathbb{T}$. The cardinality of an optimal $(t,r)$ broadcast on a finite graph $G$ is called the $(t,r)$ broadcast domination number of $G$. The concept of $(t,r)$ broadcast domination generalizes the classical problem of domination on graphs. In fact, the $(2,1)$ broadcasts on a graph $G$ are exactly the dominating sets of $G$. 

In their paper, Blessing et al. considered $(t,r)\in\{(2,2),(3,1),(3,2),(3,3)\}$ and gave optimal $(t,r)$ broadcasts on $G_{m,n}$, the grid graph of dimension $m\times n$, for small values of $m$ and $n$. They also provided upper bounds on the optimal $(t,r)$ broadcast numbers for grid graphs of arbitrary dimensions. In this paper, we define the \emph{density} of a $(t,r)$ broadcast, which allows us to provide optimal $(t,r)$ broadcasts on the infinite grid graph for all $t\geq2$ and $r=1,2$, and bound the density of the optimal $(t,3)$ broadcast for all $t\geq2$. In addition, we give a family of counterexamples to the conjecture of Blessing et al. that the optimal $(t,r)$ and $(t+1, r+2)$ broadcasts are identical for all $t\geq1$ and $r\geq1$ on the infinite grid.
\end{abstract}

\section{Introduction}


Let $G$ be a graph with vertex set $V$ and edge set $E$. Then the domination number of $G$, denoted $\delta(G)$, is the cardinality of the smallest set $B \subseteq V$, such that every vertex $v \in V$ is adjacent to at least one vertex $b \in B$. The problem of determining the domination number of \emph{finite grid graphs}  of dimension $n\times m$,  denoted $G_{m,n} = (V,E)$ with
\begin{align*}
V &= \{v_{i, j} \ :\ 1\leq i\leq n, 1\leq j\leq m\} , \\
E &= \{(v_{i,j}, v_{i+1,j}), (v_{i,j}, v_{i,j+1}) \ :\ 1 \leq i < m, 1 \leq j < n \} ,
\end{align*}
began in 1984 with the work of Jacobson and Kinch who established $\delta(G_{m,n})$ for all $n\geq1$ and $m=2,3,4$ \cite{jacobson1984domination}. This work was extended to all $n\geq1$ and $m=5,6$ in 1993 by Chang and Clark \cite{chang1993domination}. In his PhD thesis, Chang conjectured \cite{chang1992domination} that for every $16 \leq n \leq m$, 
\[ \delta(G_{m,n}) = \left \lceil{\frac{(n+2)(m+2)}{5}}\right \rceil.\]
In 2011, Gon{\c{c}}alves, Pinlou, Rao, and Thomass{\'e} confirmed Chang's conjecture thereby establishing the domination numbers of $G_{m,n}$ for all $n\geq1$ and $m \geq 6$ \cite{gonccalves2011domination}.


Since this foundational work, many have explored variations of the graph domination problem by considering vertices that ``dominate" more than their immediate neighbors. In 1990, Griggs and Hutchinson defined the \emph{$r$-domination number} of a graph $G$ as the cardinality of the smallest set $B \subseteq V$, such that for all vertices $v \in V$ there exists a vertex $b \in B$ with $dist(v, b) \leq r$ \cite{griggs1992r}. In 2002, Dunbar, Erwin, Haynes, Hedetniemi and Hedetniemi introduced the notion of broadcasts on graphs. A \emph{broadcast} is a function that assigns a signal strength $t_i\in\mathbb{Z}_{\geq0}$ to each vertex $v_i$ in a graph, and a dominating broadcast is one that ensures that for every vertex $v_i \in V$, there exists a vertex $v_j \in V$, such that $t_j - dist(v_i, v_j) > 0$ \cite{dunbar2006broadcasts}. For our purposes, we refer to each vertex $v_i \in V$ with $t_i > 0$ as a \emph{tower} of signal strength $t$. The \emph{cost} of a broadcast is defined as $\Sigma_i \ t_i$, and the $r$-domination number of a graph is proportional to the minimum cost of a broadcast in which all towers have signal strength $r$.

The problem of finding minimal cost broadcasts with various constraints on signal strength has received substantial attention in recent years. In 2014 and 2015, Koh and Soh considered the broadcast domination numbers of graph products of paths and torii \cite{soh2014broadcast, soh2015broadcast}. In 2013, Fata, Smith, and Sundaram established upper bounds on minimum broadcast costs with fixed signal strength for $G_{m, n}$, and these bounds were subsequently improved by Grez and Farina \cite{fata2013distributed, grez2014new}.

In 2015, Blessing, Insko, Johnson, and Mauretour generalized the graph domination problem further by introducing the concept of a $(t,r)$ broadcast on a graph for positive integers $t$ and $r$ \cite{blessing2015t}. They define the \emph{signal} that a vertex $v$ receives from a tower $T$ of strength $t$ as $max(t - dist(T, v), 0)$ and a $(t,r)$ broadcast as a set of towers $\mathbb{T} \subseteq V$, such that the sum of all signal received at each vertex is at least $r$.  A $(t,r)$ broadcast is \emph{optimal} if $|\mathbb{T}|$ is minimized, and on finite graphs the cardinality of an optimal $(t,r)$ broadcast $\mathbb{T}$ is called the $(t,r)$ broadcast domination number. The problem of finding an optimal $(t,r)$ broadcast specializes the optimal broadcast problem by holding signal strength constant and generalizes it by specifying a minimum amount of signal required by each vertex in the graph. We remark that the $(2,1)$ broadcast domination problem is equivalent to the classical domination problem.

This paper continues the work of Blessing et al. by extending the notion of an optimal $(t,r)$ broadcast to the integer lattice $\mathbb{Z}\times\mathbb{Z}$, which we refer to as the infinite grid and denote by $G_\infty$. To extend the notion of optimality to $G_\infty$, we generalize the concept of a $(t,r)$ broadcast domination number by considering the \emph{density} of a broadcast, defined intuitively as the proportion of the vertices of $G_{\infty}$ contained in an infinite broadcast $\mathbb{T}$. To make this notion precise, we give the following definitions.

Given a $(t,r)$ broadcast $\mathbb{T}$ on $G_\infty$, consider the vertex set $V$ of the subgraph $G_{2n+1, 2n+1}$ with its central vertex located at $(0, 0)$. Then the broadcast density of a $(t,r)$ broadcast on $G_\infty$ is defined as $\lim_{n \to \infty} \frac{|\mathbb{T} \cap V|}{|V|}$.
%
Given any graph $G$ and fixed positive integers $t$, $r$, the \emph{optimal density} of a $(t,r)$ broadcast on $G$, denoted $\delta_{t,r}(G)$, is defined as the minimum broadcast density over all $(t,r)$ broadcasts. We say that a $(t,r)$ broadcast is optimal if its density is optimal. 

With these definitions at hand, we now state our main results.


\newtheorem*{theorem:r=1}{Theorem \ref{thm_r=1}}
\begin{theorem:r=1}
 If $t\geq1$, then $\delta_{t,1}(G_{\infty}) = \frac{1}{2t^2 - 2t + 1}$.
\end{theorem:r=1}

\newtheorem*{theorem:r=2}{Theorem \ref{thm_r=2}}
\begin{theorem:r=2}
If $t>2$, then $\delta_{t,2}(G_{\infty}) = \frac{1}{2(t-1)^2}$.
\end{theorem:r=2}

\newtheorem*{theorem:conj}{Theorem \ref{thm_conj}}
\begin{theorem:conj}
If $t>2$, then $\delta_{t,3}(G_\infty) \leq \delta_{t-1,1}(G_\infty)$.
\end{theorem:conj}

Theorem 3 supports the conjecture of Blessing et al. that the optimal $(t,r)$ and $(t+1,r+2)$ broadcasts are identical in the special case $t>2$ and $r=1$. However, as we show, the conjecture is false in general. 

This paper is organized as follows. In Section 2, we define the working notions of signal and excess and introduce key concepts necessary to investigate optimal $(t, r)$ broadcasts on $G_\infty$. In Section 3, we prove Theorems \ref{thm_r=1} and \ref{thm_r=2}. In Section 4, we prove Theorem 3, demonstrate that  $\delta_{1,1}(G_\infty)\neq\delta_{2,3}(G_\infty)$, and provide a family of counterexamples showing that, in general, $\delta_{t,r}(G_\infty)\neq\delta_{t+1,r+2}(G_\infty)$ when $r>1$. We end by presenting the following revised conjecture.

\begin{conjecture}\label{conj_t=3}
For all $t>2$, $\delta_{t,3}(G_\infty) = \delta_{t-1,1}(G_\infty)$.
\end{conjecture}

\section{Preliminaries}

The signal that a vertex $v$ receives from a tower $T$ with signal strength $t$ is typically defined as $max(t - dist(v, T),0)$, a function that decreases linearly as the distance between $v$ and $T$ increases. However, if $t > r$, vertices near towers receive signal greater than $r$. We refer to this extra signal as \emph{innate excess}. Because including innate excess in calculations is cumbersome and unnecessary in practice, we introduce the following working definitions of signal and excess.

\begin{definition}
\sloppy Let $\mathbb{T}$ be a $(t,r)$ broadcast of $G$. The \emph{signal} vertex $v$ receives from a tower at vertex $T$ with signal strength $t$ is defined as $sig(v,T)=min(max(t-dist(v, T),0), r)$. The \emph{total signal} of $v$ is defined as $sig(v)= \sum_{T \in \mathbb{T}} sig(v,T)$.
\end{definition}

 \begin{definition}
The \emph{excess} of a vertex $v$ is defined as $ex(v) = sig(v) - r$.
\end{definition}

On $G_\infty$, an optimal $(t,r)$ broadcast often consists of a pattern of towers that repeats across the infinite grid. To efficiently refer to such a pattern, we introduce the following.


\begin{definition}
The \emph{standard pattern} is defined by the positive integers $d$ and $e$ as
\[p(d,e) = \{(dx+ey, y) : x,y \in \mathbb{Z}\}.\]
\end{definition}

We use the notation $\mathbb{T}(d,e)$ to refer to the set of vertices on $G_\infty$ corresponding to the standard pattern $p(d,e)$. Note that $\mathbb{T}(d,e)$ is not necessarily a $(t,r)$ broadcast for any values of $t$ and $r$. When $\mathbb{T}(d,e)$ is in fact a broadcast we call it the \emph{standard broadcast}.
Although the standard patterns represent a limited set of potential $(t,r)$ broadcasts, the simplicity of the definition provides certain practical advantages. First, the density of a standard pattern $p(d,e)$ is inversely proportional to $d$. Second, whether or not a vertex set $\mathbb{T}(d,e)$ is a standard broadcast can be checked in time proportional to $d$ \cite{DHRprogram}.

Figures \ref{(3,1)-optimal} and \ref{(2,2)-optimal}  illustrate the standard patterns $p(13,5)$ and $p(3,2)$, respectively. Note that the standard broadcast $\mathbb{T}(d,e)$ has broadcast density $\frac{1}{d}$, as one out of every $d$ vertices on each horizontal path of $G_\infty$ is included in $\mathbb{T}(d,e)$. 

\begin{figure}[h!]
\centering
\begin{tikzpicture}[scale=0.5]
    \grid{6}{6}
    \tower{3}{3}{3}{$T$}
\end{tikzpicture}
\caption{ A tower $T$ with $t=4$ and its broadcast outline.}
\label{fig_broadcastoutline}
\end{figure}
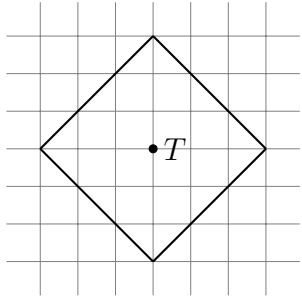

Our optimality proofs depend on the geometric structure of the infinite grid, so we define some additional terms to make our approach precise. Consider the diamond shape formed around a tower $T$ by connecting, with line segments, each vertex in the set $\{v : sig(v,T) = 1\}$ to its two nearest neighbors. We refer to this shape as the \emph{broadcast outline} of $T$. Figure \ref{fig_broadcastoutline} illustrates the broadcast outline of a tower $T$ with $t$=4. Note that the broadcast outline of a tower with signal strength $t$ encloses an area of $2(t-1)^2$ in the infinite grid $\mathbb{Z}\times\mathbb{Z}$.

The union of all grid square diagonals yields a lattice offset from the infinite grid by $45^\circ$. 
We refer to a grid square on this lattice as a \emph{half-square} because every such square enclosed by these lines has area $\frac{1}{2}$, and we refer to this new lattice itself as the \emph{half-square grid}. Because every broadcast outline falls on the half-square grid, we can count areas inside and outside of broadcast outlines using half-squares. Note that every two adjacent vertices in $G_\infty$ uniquely determine a half-square, as illustrated in Figure \ref{halfsquare}.

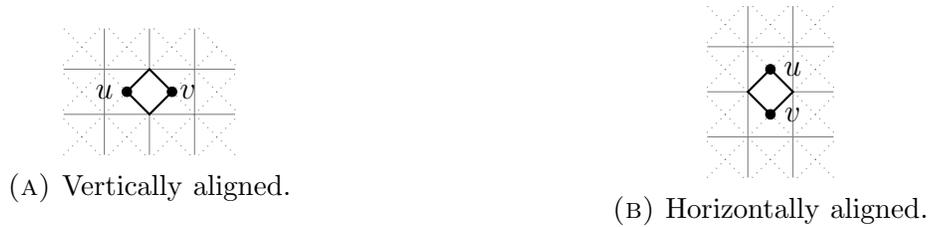
\begin{figure}[h!]
    \centering
    
    \begin{subfigure}{0.5\textwidth}
    \centering
    \begin{tikzpicture}[scale=0.6]
        \halfsquaregrid{2}{1}
        \horizhalfsquare{1}{0}{}{}
        \vertex[black]{0.5}{0.5}{}
        \vertex[black]{1.5}{0.5}{}
        \node at (0, 0.5) {{$u$}};
        \node at (1.75, 0.5) {{$\;v$}};
    \end{tikzpicture}
    \caption { Vertically aligned.}
    \end{subfigure}%
    \begin{subfigure}{0.5\textwidth}
    \centering
    \begin{tikzpicture}[scale=0.6]
        \halfsquaregrid{1}{2}
        \verthalfsquare{0}{1}{}{}
        \vertex[black]{0.5}{0.5}{}
        \vertex[black]{0.5}{1.5}{}
        \node at (0.8, 1.5) {{$\;\;u$}};
        \node at (0.8, 0.5) {{$\;\;v$}};
    \end{tikzpicture}
    \caption{ Horizontally aligned. }
    \end{subfigure}
    
    \caption{Half-squares on the half-square grid, uniquely determined by vertices $u$ and $v$.}
    \label{halfsquare}
\end{figure}

Let $h_{u,v}$ be the half-square determined by the vertices $u$ and $v$. We say that a tower $T$ \emph{covers $h_{u,v}$ to depth $min(sig(u,T), sig(v,T))$}. The \emph{depth of $h_{u,v}$} is the sum of the depth to which it is covered by every tower. We define a \emph{hole of depth $d$} as a region of contiguous half-squares covered to depth $r-d$. With these definitions at hand, we are prepared to prove our results.

\section{Optimal $(t,r)$ broadcasts on the Infinite Grid}

There exists a family of broadcasts with non-overlapping broadcast outlines that cover $G_\infty$ efficiently to depth 1. Hence, proving Theorem \ref{thm_r=1} is straightforward. 

First, observe that a tower $T$ with signal strength $t$ transfers signal to $2t^2 - 2t + 1$  vertices on the infinite grid. This follows from the fact that the vertices covered by $T$ can be grouped into two squares of dimensions $t\times t$ and $(t-1)\times (t-1)$. Figure \ref{fig_twosquarescovered} illustrates such a grouping for $t=3$.

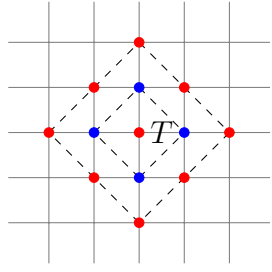
\begin{figure}[h!]
    \begin{center}
        \begin{tikzpicture}[scale = 0.6]
            \grid[gray]{4}{4}
            \draw[dashed] (0,2) -- (2,4);
            \draw[dashed] (2,4) -- (4,2);
            \draw[dashed] (4,2) -- (2,0);
            \draw[dashed] (2,0) -- (0,2);
            \draw[dashed] (2,1) -- (1,2);
            \draw[dashed] (2,1) -- (3,2);
            \draw[dashed] (1,2) -- (2,3);
            \draw[dashed] (2,3) -- (3,2);
            \node at (2.5,2) {$T$};
            \vertex[red]{2}{4}{}
            \vertex[red]{1}{3}{}
            \vertex[red]{3}{3}{}
            \vertex[red]{0}{2}{}
            \vertex[red]{2}{2}{}
            \vertex[red]{4}{2}{}
            \vertex[red]{1}{1}{}
            \vertex[red]{3}{1}{}
            \vertex[red]{2}{0}{}
            \vertex[blue]{2}{3}{}
            \vertex[blue]{1}{2}{}
            \vertex[blue]{3}{2}{}
            \vertex[blue]{2}{1}{}
        \end{tikzpicture}
        \caption{ Vertices that receive signal from a tower of strength $3$, grouped by color into a red $3\times 3$ square and a blue $2\times 2$  square. }
        \label{fig_twosquarescovered}
    \end{center}
\end{figure}

Because in any $(t,1)$ broadcast all vertices $v\in G_\infty$ satisfy $sig(v) \geq 1$ , we have that $\frac{1}{2t^2 - 2t + 1}$ is a lower bound on the optimal $(t,1)$ broadcast density. We will show that the standard broadcast 
\[\mathbb{T}^{*}_{t,1} := \mathbb{T}(2t^2-2t+1,2t-1)\] 
is a $(t,1)$ broadcast of density $\frac{1}{2t^2 - 2t + 1}$, which will establish Theorem \ref{thm_r=1}. For example, Figure \ref{(3,1)-optimal}  illustrates~$\mathbb{T}^{*}_{3,1}$.

\begin{figure}[h!]
    \begin{center}
        \begin{tikzpicture}[scale=0.6]
            \clip (-0.9, -0.9) rectangle (14.9, 6.9);
            \grid[lightgray]{14}{6}
            \tower{-2}{4}{2}{}
            \tower{0}{7}{2}{}
            \tower{-1}{-1}{2}{}
            \tower{1}{2}{2}{}
            \tower{3}{5}{2}{}
            \tower{5}{8}{2}{}
            \tower{4}{0}{2}{}
            \tower{6}{3}{2}{}
            \tower{8}{6}{2}{}
            \tower{13}{7}{2}{}
            \tower{11}{4}{2}{}
            \tower{9}{1}{2}{}
            \tower{7}{-2}{2}{}
            \tower{14}{2}{2}{}
            \tower{12}{-1}{2}{}
            \tower{16}{5}{2}{}
        \end{tikzpicture}
        \caption{ The standard broadcast $\mathbb{T}^{*}_{3,1} = \mathbb{T}(13, 5)$. }
        \label{(3,1)-optimal}
    \end{center}
\end{figure}
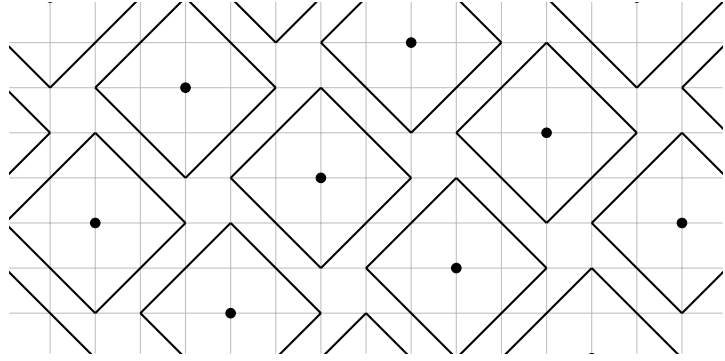

To see why $\mathbb{T}^{*}_{t,1}$ is optimal, we must understand the structure of the standard broadcast. To that end, we establish the following result.

\begin{lemma}
\label{lemma_vectorspace}
The standard pattern $p(2t^2-2t+1,2t-1)$ is the integral linear span of the vectors $\langle t,1-t\rangle$ and $\langle t-1,t\rangle$.
\end{lemma}

\begin{proof}
\sloppy By the definition of a standard pattern, $(i,j) \in p(2t^2-2t+1,2t-1)$ if ${i=(2t^2-2t+1)x+(2t-1)y}$ for some $x,y \in \mathbb{Z}$. Letting $x=1$, $y=1-t$, we have that $i=t$ and thus $(t,1-t)\in p(2t^2-2t+1,2t-1)$. Letting $x=1$, $y=t$, we have that $i=t-1$ and thus $(1-t,t)\in p(2t^2-2t+1,2t-1)$. Moreover, as a consequence of the fact that $(t,1-t), (t-1,t) \in p(2t^2-2t+1,2t-1)$, 
$(at+b(1-t),a(t-1)+bt) \in p(2t^2-2t+1,2t-1)$
for all $a, b \in \mathbb{Z}$. 
\end{proof}

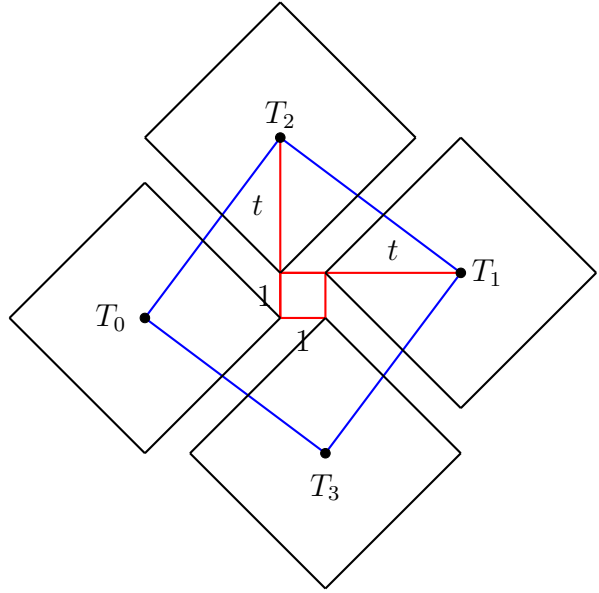
\begin{figure}[h!]
    \begin{center}
        \begin{tikzpicture}[scale = 0.6]
            \draw[-, thick, blue](2,5) -- (6,2);
            \draw[-, thick, blue](5,9) -- (9,6);
            \draw[-, thick, blue](2,5) -- (5,9);
            \draw[-, thick, blue](6,2) -- (9,6);
            \draw[-, thick, red](5,5) -- (5,6);
            \draw[-, thick, red](5,5) -- (6,5);
            \draw[-, thick, red](5,6) -- (6,6);
            \draw[-, thick, red](6,5) -- (6,6);
            \draw[-, thick, red](5,9) -- (5,5);
            \draw[-, thick, red](5,6) -- (9,6);
            \node at (4.5, 7.5) {$t$};
            \node at (7.5, 6.5) {$t$};
            \node at (4.65, 5.5) {1};
            \node at (5.5, 4.5) {1};
            \tower{2}{5}{3}{}
            \node at (1.25,5) {{$T_0$}};
            \tower{9}{6}{3}{$T_1$}
            \tower{5}{9}{3}{}
            \node at (5,9.5) {{$T_2$}};
            \tower{6}{2}{3}{}
            \node at (6,1.25) {{$T_3$}};
        \end{tikzpicture}
        \caption{ A square region bounded by four towers under the broadcast $\mathbb{T}^{*}_{t,1}$. Every vertex within the region is covered to depth at least $1$. }
        \label{fig_r=1proof}
    \end{center}
\end{figure}

\begin{theorem}
\label{thm_r=1}
 If $t\geq1$, then $\delta_{t,1}(G_{\infty}) = \frac{1}{2t^2 - 2t + 1}$.
\end{theorem}

\begin{proof}
As a consequence of Lemma \ref{lemma_vectorspace}, $G_\infty$ under $\mathbb{T}^{*}_{t,1}$ can be divided into square regions bounded at their corners by towers as illustrated in blue in Figure \ref{fig_r=1proof}. Every vertex $v$ in $G_\infty$ is located inside or on the broadcast outline of exactly one tower and thus $sig(v)\geq1$. Thus every vertex in $G_\infty$ under $\mathbb{T}^{*}_{t,1}$ is covered to depth at least 1.
\end{proof}

Before proving Theorem \ref{thm_r=2}, the analogous result for $(t,2)$ broadcasts with $t>2$, we note a degenerate case.

\begin{proposition}
\label{prop_2,2_optimal}
$\delta_{2,2}(G_{\infty})=\frac{1}{3}$.
\end{proposition}

\begin{proof} As illustrated in Figure \ref{(2,2)-optimal}, the standard broadcast $\mathbb{T}(3,2)$ is a (2,2)-broadcast. Moreover, $ex(v) = 0$ at every vertex $v$ and thus the broadcast density cannot be decreased. As $\mathbb{T}(3,2)$ has density $\frac{1}{3}$, $\delta_{2,2}(G_{\infty})=\frac{1}{3}$.
\end{proof}

\begin{figure}[h!]
    \begin{center}
        \begin{tikzpicture}[scale=0.6]
            \clip (-0.9, -0.9) rectangle (10.9, 3.9);
            \grid[lightgray]{10}{3}
            \tower{-1}{4}{1}{}
            \tower{0}{3}{1}{}
            \tower{2}{4}{1}{}
            \tower{-1}{1}{1}{}
            \tower{1}{2}{1}{}
            \tower{3}{3}{1}{}
            \tower{5}{4}{1}{}
            \tower{0}{0}{1}{}
            \tower{2}{1}{1}{}
            \tower{4}{2}{1}{}
            \tower{6}{3}{1}{}
            \tower{8}{4}{1}{}
            \tower{1}{-1}{1}{}
            \tower{3}{0}{1}{}
            \tower{5}{1}{1}{}
            \tower{7}{2}{1}{}
            \tower{9}{3}{1}{}
            \tower{11}{4}{1}{}
            \tower{4}{-1}{1}{}
            \tower{6}{0}{1}{}
            \tower{8}{1}{1}{}
            \tower{10}{2}{1}{}
            \tower{7}{-1}{1}{}
            \tower{9}{0}{1}{}
            \tower{11}{1}{1}{}
            \tower{10}{-1}{1}{}
        \end{tikzpicture}
        \caption{ The standard broadcast $\mathbb{T}(3, 2)$ with towers of strength $t=2$.}
        \label{(2,2)-optimal}
    \end{center}
\end{figure}
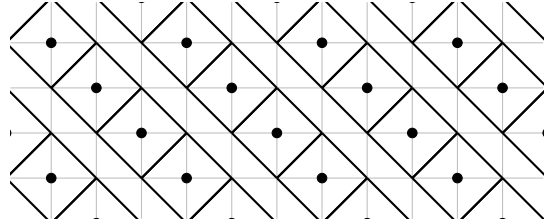

Having established the degenerate case, we proceed to the proof of Theorem \ref{thm_r=2}. The proof employs the method of contradiction and follows immediately from several technical results. In Lemma \ref{lemma_t=2upperbound}, we will establish the existence of a $(t, 2)$ broadcast with density $\frac{1}{2(t-1)^2}$ for any $t>2$. Then in Lemmas \ref{lemma_holes}-\ref{lemma_moreoverlapthanholes} we will  show the impossibility of a $(t, 2)$ broadcast with $t>2$ and broadcast density less than $\frac{1}{2(t-1)^2}$, thus establishing that, for any $t>2$, $\delta_{t,2}(G_\infty)=\frac{1}{2(t-1)^2}$, as desired.

\begin{figure}[h!]
    \centering
    
    \begin{subfigure}{0.5\textwidth}
    \centering
    \begin{tikzpicture}[scale=0.6]
        \clip (-1, -1) rectangle (11, 6);
        \grid[lightgray]{10}{5}
        \tower{0}{4}{2}{}
        \tower{2}{6}{2}{}
        \tower{0}{0}{2}{}
        \tower{2}{2}{2}{}
        \tower{4}{4}{2}{}
        \tower{6}{6}{2}{}
        \tower{4}{0}{2}{}
        \tower{6}{2}{2}{}
        \tower{8}{4}{2}{}
        \tower{10}{6}{2}{}
        \tower{8}{0}{2}{}
        \tower{10}{2}{2}{}
        \tower{12}{0}{2}{}
    \end{tikzpicture}
    \caption { Regular tiling.}
    \end{subfigure}%
    \begin{subfigure}{0.5\textwidth}
    \centering
    \begin{tikzpicture}[scale=0.6]
        \clip (-1, -1) rectangle (11, 6);
        \grid[lightgray]{10}{5}
        \tower{-1}{3}{2}{}
        \tower{1}{5}{2}{}
        \tower{0}{0}{2}{}
        \tower{2}{2}{2}{}
        \tower{4}{4}{2}{}
        \tower{6}{6}{2}{}
        \tower{3}{-1}{2}{}
        \tower{5}{1}{2}{}
        \tower{7}{3}{2}{}
        \tower{9}{5}{2}{}
        \tower{8}{0}{2}{}
        \tower{10}{2}{2}{}
        \tower{11}{-1}{2}{}
    \end{tikzpicture}
    \caption{ Offset tiling. }
    \end{subfigure}
    
    \caption{ Two regular tilings of $G_\infty$ using broadcast outlines with $t=3$.}
    \label{regular_tiling}
\end{figure}
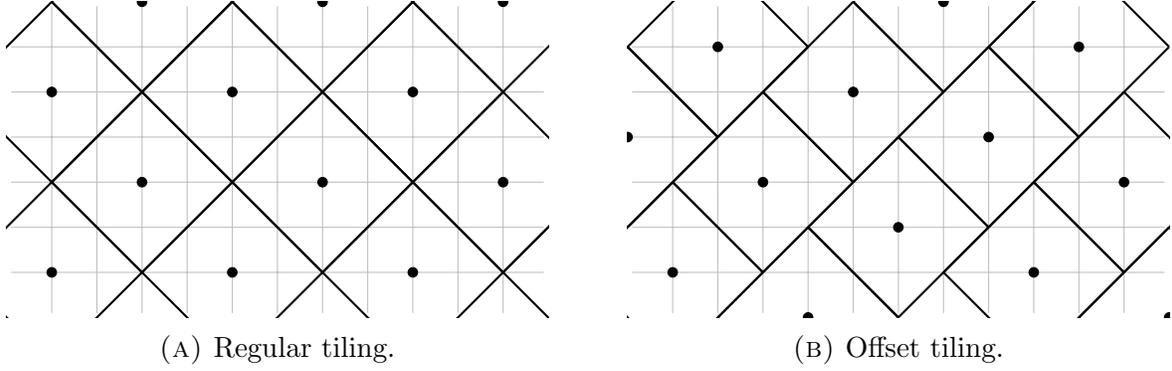

\begin{lemma}
\label{lemma_t=2upperbound}
For any $t>2$, there exists a $(t,2)$ broadcast with density $\frac{1}{2(t-1)^2}$.
\end{lemma}

\begin{proof}
Let $\mathbb{T}^{*}_{t,2}$ be a set of towers corresponding to a regular tiling of the infinite grid using broadcast outlines, as illustrated in Figure \ref{regular_tiling}. In $\mathbb{T}^{*}_{t,2}$, every vertex $v\in G_\infty$ is either inside a broadcast outline or on the perimeter of multiple broadcast outlines. If $v$ is inside a broadcast outline, there exists a tower $T$ such that $sig(v,T)\geq2$. If $v$ is on the perimeter of multiple broadcast outlines, there exist towers $T_0$ and $T_1$ such that $sig(v,T_0)+sig(v,T_1)=2$. Thus $sig(v)\geq2$ for all $v\in \mathbb{T}^{*}_{t,2}$ and $\mathbb{T}^{*}_{t,2}$ is a $(t,2)$ broadcast on the infinite grid for all $t>2$.

As the broadcast outline of a tower $T$ with strength $t$ has side length $\sqrt{2}(t-1)$, the area within $T$'s broadcast outline is $2(t-1)^2$ grid squares. Because the broadcast outlines of $\mathbb{T}^{*}_{t,2}$ tile the infinite grid, there exists exactly one tower in $\mathbb{T}^{*}_{t,2}$ for every $2(t-1)^2$ grid squares. We can create a 1-to-1 mapping between vertices in $G_\infty$ and grid squares by associating each vertex $v$ in $G_\infty$ with the grid square above and to the right of $v$. As there exists one tower in $\mathbb{T}^{*}_{t,2}$ for every $2(t-1)^2$ grid squares, one in every $2(t-1)^2$ vertices in $\mathbb{T}^{*}_{t,2}$ is a tower and $\mathbb{T}^{*}_{t,2}$ has broadcast density $\frac{1}{2(t-1)^2}$.
\end{proof}

Ultimately, we show that $\mathbb{T}^{*}_{t,2}$ is an optimal $(t,2)$ broadcast. In order to do this, we show that any $(t,2)$ broadcast with density less than that of $\mathbb{T}^{*}_{t,2}$  would contain holes and that every $(t,2)$ broadcast with holes has broadcast density greater than that of $\mathbb{T}^{*}_{t,2}$, thus establishing a contradiction.

\begin{lemma}
For $t>2$, any $(t,2)$ broadcast with density less than $\frac{1}{2(t-1)^2}$ contains a hole of depth $2$.
\label{lemma_holes}
\end{lemma}

\begin{proof}
Recall that holes of depth $d$ are contiguous regions of half-squares covered to depth $r-d$. In a $(t, 2)$ broadcast, a hole of depth $2$ is thus a region of contiguous, uncovered half-squares bounded by several broadcast outlines. 

In $\mathbb{T}^{*}_{t,2}$, every tower $T$ uniquely covers exactly $2(t-1)^2$ grid squares or, equivalently, $4(t-1)^2$ half-squares. Any $(t,2)$ broadcast with broadcast density less than $\mathbb{T}^{*}_{t,2}$ must span the same area with fewer towers, and thus must leave some half-square uncovered, creating a hole of depth~2.
\end{proof}

For the remainder of this section, we refer to holes of depth 2 simply as holes. Furthermore, we refer to a half-square that is part of a hole as a \emph{hole half-square}, a half-square that is not part of a hole as a \emph{covered half-square}, and a half-square that is covered by at least two towers as an \emph{overlap half-square}.

\begin{lemma}
For $t>2$, any hole of depth $2$ in a $(t,2)$ broadcast is convex.
\label{lemma_holesareconvex}
\end{lemma}
\begin{proof}
We define a \emph{spur point} as a vertex of $G_\infty$ adjacent to three hole half-squares and one covered half-square. Figure \ref{spurpoint} illustrates a spur point. We define a \emph{convex hole} as any hole that does not have a spur point on its boundary. 

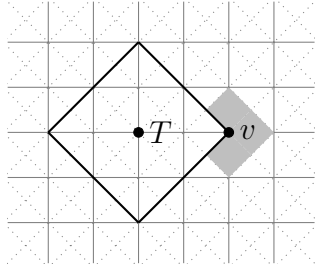
\begin{figure}[h!]
    \begin{center}
        \begin{tikzpicture}[scale=0.6]
            \halfsquaregrid[]{5}{4}
            \holehshoriz[fill=lightgray]{4}{1}
            \holehshoriz[fill=lightgray]{4}{2}
            \holehsvert[fill=lightgray]{4}{2}
            \tower{2}{2}{2}{$T$}
            \vertex[]{4}{2}{$v$}
        \end{tikzpicture}
        \caption{A spur point $v$ adjacent to a non-convex hole with area shaded in grey. }
        \label{spurpoint}
    \end{center}
\end{figure}

Suppose there exists a non-convex hole in a $(t,2)$ broadcast, and let $v$ be a spur point on its boundary.  Because $v$ is adjacent to exactly one covered half-square, it can receive at most a signal of 1 from at most one tower $T$, as illustrated in Figure \ref{spurpoint}. As $sig(v)$ must be greater than or equal to 2 in a $(t,2)$ broadcast, this yields a contradiction. Therefore any hole of depth 2 must be convex.
\end{proof}

It follows from Lemmas \ref{lemma_holes} and \ref{lemma_holesareconvex} that any $(t,2)$ broadcast with density less than that of $\mathbb{T}^{*}_{t,2}$ contains a convex hole. The following result constrains the possible shape of a convex hole that might appear in a $(t,2)$ broadcast of $G_\infty$.

\begin{lemma}
For $t>2$, there exists no convex hole in a $(t,2)$ broadcast of half-square dimension $m\times n$ with $m>2$ and $n\geq2$.
\label{lemma_threeholetypes}
\end{lemma}
\begin{proof}
As a consequence of Lemma \ref{lemma_holesareconvex}, for $t>2$, holes in a $(t,2)$ broadcast must be rectangular or an infinite strip with a width of one half-square. We say that a rectangular hole has dimension $m \times n$ if $m$ and $n$ are its dimensions in half-squares. Figure \ref{fig2x2} displays the two convex holes of dimension $2 \times 2$. The hole displayed in Figure \ref{fig2x2}A is impossible in a $(t,2)$ broadcast because it encloses a vertex $v$ with $sig(v)=0$; thus the only possible $2 \times 2$ hole is of the type shown in Figure \ref{fig2x2}B.

\begin{figure}[h!]
    \centering
    
    \begin{subfigure}{0.5\textwidth}
    \centering
    \begin{tikzpicture}[scale=0.6]
        \halfsquaregrid[]{2}{2}
        \fill[lightgray, rotate around={45: (0,1)}] (0, 1) rectangle (1.414, -0.414);
        \vertex[]{1}{1}{$v$}
        \vertex[]{0}{1}{}
        \vertex[]{1}{0}{}
        \vertex[]{1}{2}{}
        \vertex[]{2}{1}{}
    \end{tikzpicture}
    \caption{ An impossible $2\times2$ hole. }
    \end{subfigure}%
    \begin{subfigure}{0.5\textwidth}
    \centering
    \begin{tikzpicture}[scale=0.6]
        \halfsquaregrid[]{1}{1}
        \fill[lightgray, rotate around={45: (-0.5,0.5)}] (-0.5, 0.5) rectangle (0.914, -0.914);
        \vertex[]{0}{1}{}
        \vertex[]{1}{0}{}
        \vertex[]{0}{0}{}
        \vertex[]{1}{1}{}
    \end{tikzpicture}
    \caption { A possible $2\times2$ hole.}
    \end{subfigure}%
    
    \caption{ The only $2\times2$ convex holes of depth 2.}
    \label{fig2x2}
\end{figure}
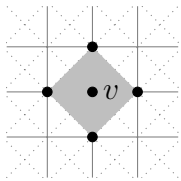
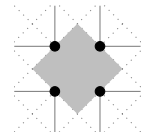

Suppose for contradiction that there exists a convex hole of dimension $n \times m$ with $n \geq  2$ and $m > 2$ in a $(t,2)$ broadcast. Such a hole must include a smaller $2 \times 2$ region, which must appear as depicted in Figure \ref{fig2x2}B. However, the $2 \times 2$ region depicted in Figure \ref{fig2x2}B cannot be extended without enclosing a vertex of $G_\infty$. Thus the $2 \times 2$ hole depicted in Figure \ref{fig2x2}B is the only possible convex hole with both dimensions greater than 1.
\end{proof}

As previously noted, a $(t,2)$ broadcast with density less than that of $\mathbb{T}^{*}_{t,2}$ cannot ensure that every half-square in $G_\infty$ is covered. However, convex holes cause inefficiencies in distributing signal because the vertices surrounding a convex hole require towers to be arranged in a manner that causes significant excess. Lemmas \ref{lemma_moreholesthanoverlap} and \ref{lemma_moreoverlapthanholes} make precise this fundamental inconsistency by proving two contradictory claims about any $(t,2)$ broadcast with density less than that of $\mathbb{T}^{*}_{t,2}$. 

For Lemmas \ref{lemma_moreholesthanoverlap} and \ref{lemma_moreoverlapthanholes}, we define the \emph{hole density} $\delta_{hole}(\mathbb{T})$ of a $(t,2)$ broadcast $\mathbb{T}$ as the proportion of hole half-squares on the half-square grid and the \emph{overlap density} $\delta_{overlap}(\mathbb{T})$ as the proportion of overlap half-squares on the half-square grid.

\begin{lemma}
For $t>2$, if $\mathbb{T}$ is a $(t, 2)$ broadcast with broadcast density less than $\frac{1}{2(t-1)^2}$, then $\delta_{hole}(\mathbb{T})>\delta_{overlap}(\mathbb{T})$.
\label{lemma_moreholesthanoverlap}
\end{lemma}

\begin{proof}
As a consequence of Lemmas \ref{lemma_holes} and \ref{lemma_holesareconvex}, any $(t, 2)$ broadcast with broadcast density less than $\frac{1}{2(t-1)^2}$ has at least one convex hole of depth 2. In addition, every hole in this $(t,2)$ broadcast is convex.

As noted in the proof of Lemma \ref{lemma_holes}, any tower $T$ can cover at most $4(t-1)^2$ half-squares. This is exactly the case for the broadcast $\mathbb{T}^{*}_{t,2}$, in which there are no hole half-squares or overlap half-squares. Because each hole half-square is not covered by any towers, and each overlap half-square is covered by at least two towers, any broadcast $P$ with $\delta_{hole}(P)\leq\delta_{overlap}(P)$ has broadcast density greater than or equal to $\frac{1}{2(t-1)^2}$. This proves the contrapositive, thereby establishing the result.
\end{proof}

\begin{lemma}
For $t>2$, if $\mathbb{T}$ is a $(t, 2)$ broadcast with broadcast density less than $\frac{1}{2(t-1)^2}$, then $\delta_{hole}(\mathbb{T})\leq\delta_{overlap}(\mathbb{T})$.
\label{lemma_moreoverlapthanholes}
\end{lemma}
\begin{proof}

As a consequence of Lemma \ref{lemma_threeholetypes}, any hole in a $(t, 2)$ broadcast with broadcast density less than $\frac{1}{2(t-1)^2}$ has dimension $2\times2$, $1 \times n \ (n \geq 1)$, or $1 \times \infty$. We define a  counting method that uniquely associates overlap half-squares to holes and show that each hole with area $A$ can be associated with a region consisting of at least $A$ overlap half-squares, which is sufficient to establish Lemma \ref{lemma_moreoverlapthanholes}.

We define the distance between two half-squares as the Euclidean distance between their center points, and we define the distance between a half-square $h$ and a hole $H$ as the shortest distance between $h$ and any individual hole half-square in $H$. 

We now \emph{associate} each overlap half-square $h$ in a $(t,2)$ broadcast with a convex hole as follows:
\begin{itemize}
    \item If $h$ is uniquely close to a single hole $H$, associate $h$ with $H$. 
    \item If $h$ is equally close to exactly two holes $H_0$ and $H_1$, associate half of $h$ with $H_0$ and half of $h$ with $H_1$. 
    \item If $h$ is equally close to more than two holes, leave $h$ unassociated.
\end{itemize}
Note that association counts the area of each overlap half-square exactly once. 

In subsequent illustrations we outline hole boundaries in black and fill regions internal to holes with gray. Regions located within the broadcast outline of a tower, referred to as \emph{tower footprints}, are outlined in red. When two tower footprints overlap, one is drawn slightly smaller than the other for clarity. Tower footprints often have the size and shape of the broadcast outline for a tower with signal strength $t=3$ because this is the smallest value allowed for $t$ in the scope of this proof. However, because each tower footprint outlines only the area necessarily covered by a tower's broadcast, the extent of a tower footprint is accurate in general. Half-squares associated with a hole are shaded in orange, and half-squares half-associated with a hole are shaded in yellow. \\

\noindent \textbf{Claim 1:} Every $2 \times 2$ hole is associated with 4 overlap half-squares.\\

\begin{figure}[h!]
    \centering
    
    \begin{subfigure}{1.0\textwidth}
    \centering
    \begin{tikzpicture}[scale=0.6]
        \clip (-0.9,0.1) rectangle (6.9,7.9);
        \halfsquaregrid[]{6}{7}
        \redsquare{2}{3};
        \figeleven{0}
    \end{tikzpicture}
    \caption { The minimum footprint required to outline the hole. }
    \end{subfigure}
    \begin{subfigure}{1.0\textwidth}
    \centering
    \begin{tikzpicture}[scale=0.6]
        \halfsquaregrid[]{21}{7}
        \ohalfsquare{3}{4}
        \ohalfsquare{3.5}{3.5}
        \ohalfsquare{3.5}{4.5}
        \ohalfsquare{4}{4}
        \yhalfsquare{2}{3}
        \redsquare{2}{3}
        \smredsquare{3}{4}
        \figeleven{0}
        
        \ohalfsquare{9}{3}
        \ohalfsquare{10.5}{4.5}
        \redsquare{8}{2}
        \redsquare{10}{4}
        \figeleven{7}
        \ohalfsquare{17}{3}
        \ohalfsquare{17.5}{3.5}
        \ohalfsquare{17.5}{2.5}
        \ohalfsquare{18}{3}
        \yhalfsquare{18.5}{4.5}
        \redsquare{16}{2}
        \smredsquare{17}{3}
        \figeleven{15}
    \end{tikzpicture}
    \caption{ Every configuration of tower footprints that ensures $sig(v)\geq2$. }
    \end{subfigure}
    
    \caption{ Possible arrangements of tower footprints around an open-faced $1 \times 2$ hole. }
    \label{fig_open1x2}
\end{figure}
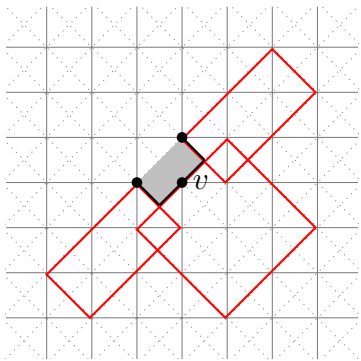
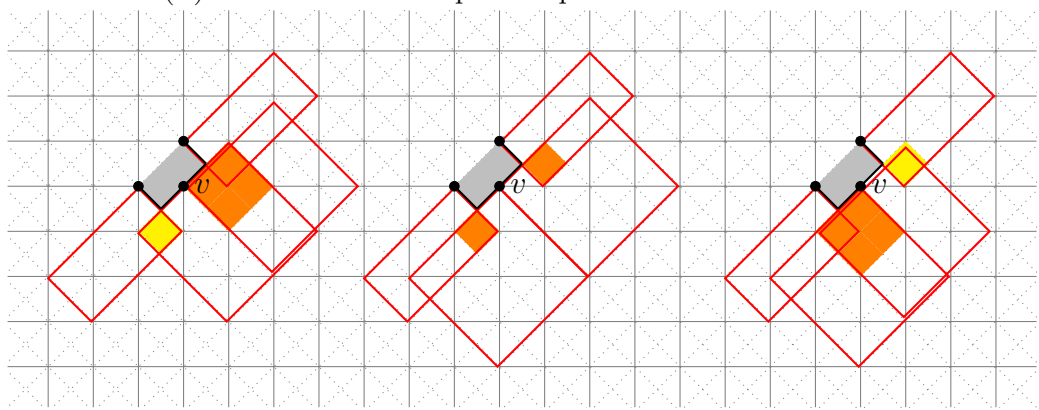

Figure \ref{fig_open1x2} depicts the possible arrangements of tower footprints around a $1 \times 2$ hole with an open face. Because the corners of the hole are located inside a grid square, at least three towers with signal strength $t\geq 2$ are required to form the hole's boundary. The minimum areas covered by the broadcast outlines of these towers are drawn in red in Figure \ref{fig_open1x2}A.

In Figure \ref{fig_open1x2}, vertex $v$ must have reception at least 2 and because it borders a hole can be located no closer to any tower $T$ than its broadcast outline. Thus $v$ must be located on at least two distinct broadcast outlines. Depending on the location of the towers, three different tower footprints are possible: $v$ can be located on the top corner, on the left corner, or on the top-left side of a tower's footprint. Figure \ref{fig_open1x2}B depicts the possible configurations of two tower footprints that form the appropriate hole shape and ensure $sig(v)\geq2$. In each tower configuration, at least 2 overlap half-squares are associated with the open $1\times2$ hole. 

Because we can combine two open $1\times2$ holes to make a $2\times 2$ hole and because we can make the same argument on both sides of the newly created $2\times2$ hole, it follows that at least 4 overlap half-squares are associated with every $2 \times 2$ hole. \\

\noindent\textbf{Claim 2.}  Every finite $1 \times n$ hole is associated with at least $n$ overlap half-squares. \\

The previous argument concerning Figure \ref{fig_open1x2} suffices to prove that the $1 \times 2$ hole is associated with at least 2 overlap half-squares.

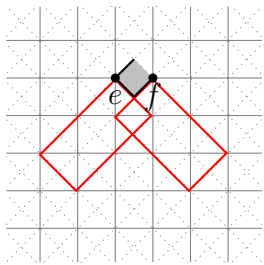
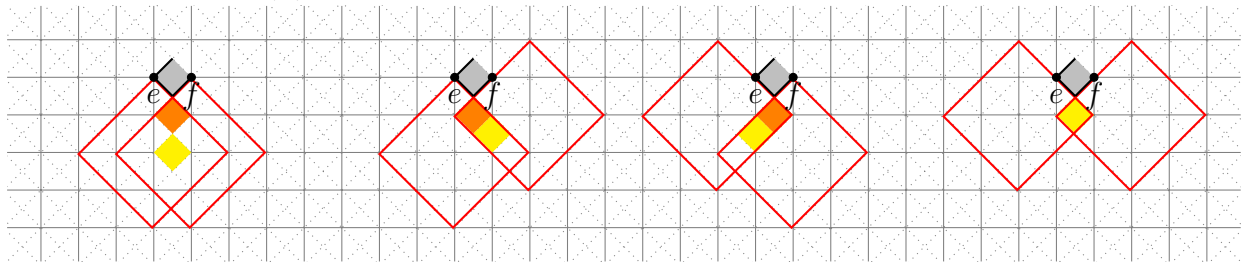
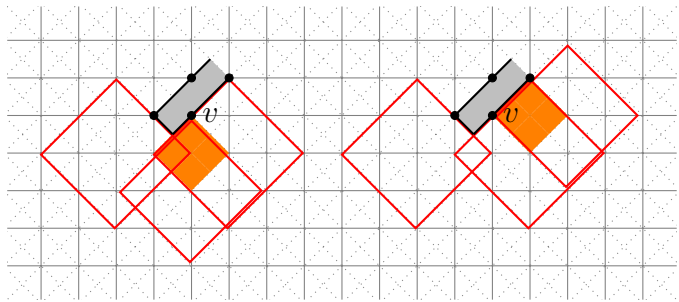
\begin{figure}[h!]
    \centering
    
    \begin{subfigure}{1.0\textwidth}
    \centering
    \begin{tikzpicture}[scale=0.5]
        \halfsquaregrid[]{5}{5}
        \redoutline[thick, red]{0}{2}{4}{2}
        \redoutline[thick, red]{2}{3}{2}{4}
        \figendbox{0}
    \end{tikzpicture}
    \caption { The minimum footprint required to outline the hole. }
    \end{subfigure}
    
    \begin{subfigure}{1.0\textwidth}
    \centering
    \begin{tikzpicture}[scale=0.5]
        \halfsquaregrid[]{31}{5}
        
        \ohalfsquare{3}{3}
        \yhalfsquare{3}{2}
        \redsquare{1}{2}
        \redsquare{2}{2}
        \figendbox{1}
        
        \ohalfsquare{11}{3}
        \yhalfsquare{11.5}{2.5}
        \redsquare{9}{2}
        \redsquare{11}{3}
        \figendbox{9}
        
        \ohalfsquare{19}{3}
        \yhalfsquare{18.5}{2.5}
        \redsquare{16}{3}
        \redsquare{18}{2}
        \figendbox{17}
        
        \yhalfsquare{27}{3}
        \redsquare{24}{3}
        \redsquare{27}{3}
        \figendbox{25}
    \end{tikzpicture}
    \caption{ Every configuration of tower footprints that outlines edges $e$ and $f$ of the hole. }
    \end{subfigure}
    
    \begin{subfigure}{1.0\textwidth}
        \centering
        \begin{tikzpicture}[scale = 0.5]
            \halfsquaregrid[]{16}{6}
            \ohalfsquare{3}{3}
            \ohalfsquare{3.5}{2.5}
            \ohalfsquare{3.5}{3.5}
            \ohalfsquare{4}{3}
            \smredsquare{2}{2}
            \figendboxthree{1}

            \ohalfsquare{12}{4}
            \ohalfsquare{12.5}{3.5}
            \ohalfsquare{12.5}{4.5}
            \ohalfsquare{13}{4}
            \smredsquare{12}{4}
            \figendboxthree{9}
        \end{tikzpicture}
        \caption{ Every possible elaboration on case four of Figure \ref{fig_end1xn}B such that $sig(v)\geq2$ in the $n\geq3$ case. }
        \end{subfigure}
    
    \caption{ Possible arrangements of tower footprints around the ending half-square of a $1 \times n$ hole. }
    \label{fig_end1xn}
\end{figure}

Figure \ref{fig_end1xn} depicts the possible arrangements of tower footprints around an ending half-square of a $1 \times n$ hole. The edges $e$ and $f$ of the $1\times n$ hole must be formed by the broadcast outlines of at least two towers. Figure \ref{fig_end1xn}A illustrates the minimum area covered by the broadcast footprints of these two towers, and Figure \ref{fig_end1xn}B illustrates every possible configuration of their two broadcast footprints.

Note that in the first three cases of Figure \ref{fig_end1xn}B, 1.5 half-squares are associated with the $1\times n$ hole. We refer to these cases as the \emph{good end configurations}. In case 4 of Figure \ref{fig_end1xn}B, which we refer to as the \emph{bad end configuration}, only 0.5 half-squares are associated with the $1\times n$ hole. Because both sides of the $1\times1$ hole necessitate an end configuration, every $1\times1$ hole is associated with at least 1 overlap half-square. 

Figure \ref{fig_end1xn}C displays every possible configuration of tower footprints around a $1\times n, n\geq3$ hole in the bad end configuration with $sig(v)\geq2$. In every case, the hole is associated with 4 overlap half-squares, and thus every $1\times3$ hole with at least one end in the bad end configuration is associated with at least 4 overlap half-squares. Moreover, every $1\times3$ hole with both ends in a good end configuration is associated with at least 3 overlap half-squares. As a consequence, we have that every $1\times3$ hole is associated with at least 3 overlap half-squares.

\begin{figure}[h!]
    \centering
    
    \begin{subfigure}{1\textwidth}
    \centering
    \begin{tikzpicture}[scale=0.6]
        \clip (-0.9,-0.9) rectangle (5.9,5.9);
        \hole[cyan]{0.5}{2.5}{4}{10}
        \redoutline{0}{2}{6}{4}
        \halfsquaregrid[]{5}{5}
        \hole{0}{3}{4}{1}
        \li{0.5}{2.5}{2.5}{4.5}
        \vertex{1}{3}{$v$}
        \vertex{2}{4}{$u$}
        
    \end{tikzpicture}
    \caption { The minimum area covered by tower outlines adjacent to an edge section of the $1\times n$ hole. The blue stripe illustrates the area in which we count associated half-squares. }
    \end{subfigure}
    
    \begin{subfigure}{1.0\textwidth}
        \centering
        \begin{tikzpicture}[scale = 0.6]
            \halfsquaregrid[]{12}{6}
            \redsquare{0}{2}
            \redsquare{1}{3}
            \ohalfsquare{1}{3}
            \ohalfsquare{1.5}{2.5}
            \ohalfsquare{1.5}{3.5}
            \ohalfsquare{2}{3}
            \hole{0}{3}{4}{1}
            \li{0.5}{2.5}{2.5}{4.5}
            \vertex{1}{3}{$u$}
            \vertex{2}{4}{$v$}
        
            \ohalfsquare{8}{4}
            \ohalfsquare{8.5}{3.5}
            \yhalfsquare{8.5}{4.5}
            \node at (9,4.3) {*};
            \yhalfsquare{9}{4}
            \node at (9.5,3.9) {*};
            \smredsquare{7}{3}
            \redsquare{8}{4}
            \redoutline{6}{2}{2}{4}
            \hole{6}{3}{4}{1}
            \li{6.5}{2.5}{8.5}{4.5}
            \vertex{7}{3}{$u$}
            \vertex{8}{4}{$v$}
        \end{tikzpicture}
        \caption{ Every possible configuration of tower footprints along a long edge segment such that $sig(v)\geq2$. }
        \end{subfigure}
    
    \caption{ Possible arrangements of broadcast footprints adjacent to an edge section of the $1\times n$ hole. Half-squares marked with an asterisk are outside the area in which we count associated half-squares. }
    \label{fig_longedge}
\end{figure}
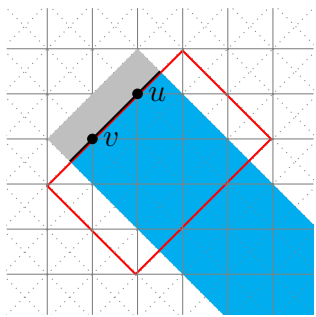
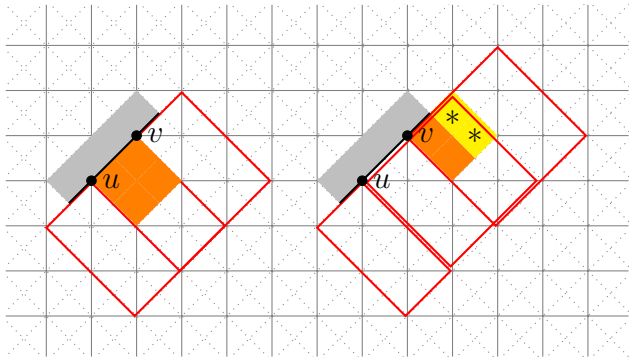

Figure \ref{fig_longedge} depicts the possible arrangements of tower footprints along the long edge of a hole of dimension $1 \times n$ with $n \geq 4$. We refer to any section of hole boundary of the type illustrated in Figure \ref{fig_longedge}A as a \emph{long edge segment}. Figure \ref{fig_longedge}A outlines the minimum area that must be covered by broadcast outlines to form the long edge segment in red.

In previous arguments, we have counted only those associated half-squares positioned at both ends of the $1\times n$ hole. When considering the long edge of the $1\times n$ hole, we count only those associated half-squares within the blue \emph{counting region} depicted in Figure~\ref{fig_longedge}A. Because each distinct long edge segment of a hole has a distinct counting region, this restriction ensures that we never double-count an associated half-square.

Figure \ref{fig_longedge}B illustrates every possible arrangement of broadcast footprints that ensures both that $sig(v) \geq 2$ and that the illustrated long edge segment containing the vertices $u$ and $v$ is defined by broadcast outlines. In every case, at least 2 overlap half-squares within the counting region are associated with the hole. 

One side of the $1\times4$ hole is a long edge segment. As a consequence, the $1 \times 4$ hole with two good end configurations is associated with at least 5 overlap half-squares. Moreover, as previously illustrated by Figure \ref{fig_end1xn}C, the $1 \times 4$ hole with at least one bad end configuration is associated with at least 4 overlap half-squares. Thus every $1\times4$ hole is associated with at least 4 overlap half squares.

\begin{figure}[h!]
    \centering
    
    \begin{subfigure}{0.33\textwidth}
        \centering
        \begin{tikzpicture}[scale=0.6]
            \halfsquaregrid{3}{4}
            \hole{0}{1}{5}{1}
            \li{0}{1}{0.5}{0.5}
            \li{0}{1}{2.5}{3.5}
            \li{0.5}{0.5}{3}{3}
            \li{2.5}{3.5}{3}{3}
            \vertex{0}{1}{}
            \vertex{1}{1}{}
            \vertex{1}{2}{}
            \vertex{2}{2}{}
            \vertex{2}{3}{}
            \vertex{3}{3}{}
            \rbracket{1}{1}{2}
            \lbracket{1}{2}{2}
            \node[blue] at (0,0) {0.5};
            \node[blue] at (3.5,3.5) {0.5};
        \end{tikzpicture}
        \caption {  }
    \end{subfigure}%
    \begin{subfigure}{0.33\textwidth}
        \centering
        \begin{tikzpicture}[scale=0.6]
            \halfsquaregrid{4}{5}
            \hole{0}{1}{6}{1}
            \li{0}{1}{0.5}{0.5}
            \li{0}{1}{3}{4}
            \li{0.5}{0.5}{3.5}{3.5}
            \li{3}{4}{3.5}{3.5}
            \vertex{0}{1}{}
            \vertex{1}{1}{}
            \vertex{1}{2}{}
            \vertex{2}{2}{}
            \vertex{2}{3}{}
            \vertex{3}{3}{}
            \vertex{3}{4}{}
            \rbracket{1}{1}{2}
            \lbracket{1}{2}{3}
            \node[blue] at (0,0) {0.5};
            \node[blue] at (4,4) {0.5};
        \end{tikzpicture}
        \caption{  }
    \end{subfigure}%
    \begin{subfigure}{0.33\textwidth}
        \centering
        \begin{tikzpicture}[scale = 0.6]
            \halfsquaregrid{4}{5}
            \hole{0}{1}{7}{1}
            \li{0}{1}{0.5}{0.5}
            \li{0}{1}{3.5}{4.5}
            \li{0.5}{0.5}{4}{4}
            \li{3.5}{4.5}{4}{4}
            \vertex{0}{1}{}
            \vertex{1}{1}{}
            \vertex{1}{2}{}
            \vertex{2}{2}{}
            \vertex{2}{3}{}
            \vertex{3}{3}{}
            \vertex{3}{4}{}
            \vertex{4}{4}{}
            \rbracket{2}{2}{3}
            \lbracket{1}{2}{3}
            \node[blue] at (0,0) {0.5};
            \node[blue] at (4.5,4.5) {0.5};
        \end{tikzpicture}
        \caption{  }
    \end{subfigure}%
    
    \caption{ Associations for the holes of dimension $1 \times n$ with $n = 5,6,7$. }
    \label{fig_n=567}
\end{figure}
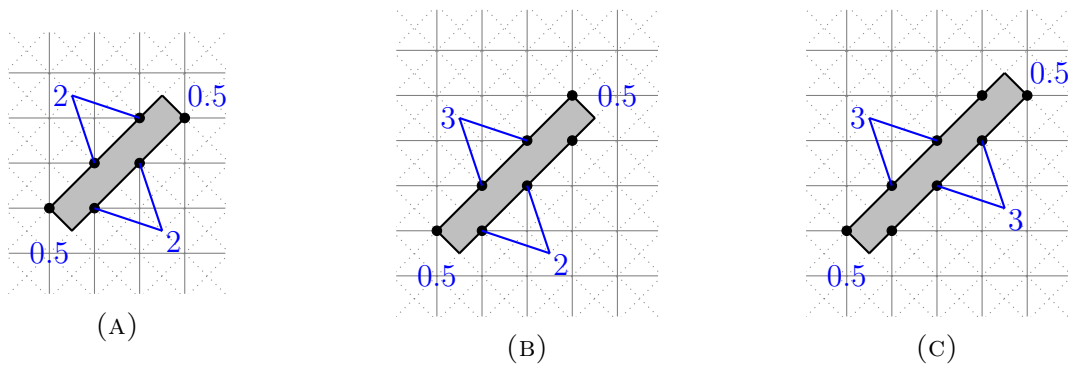

Our arguments can be extended to the $1 \times n$ case as illustrated in Figure \ref{fig_n=567}. First, regardless of the arrangement of broadcast outlines that forms the boundary of the hole, at least 0.5 half-squares at each end of the $1\times n$ hole must be associated with the hole as shown in Figure \ref{fig_end1xn}. Second, at least 2 half-squares in the associated counting region of each distinct long edge segment are associated with the hole. This is enough to demonstrate that every $1\times5$ hole is associated with at least 5 overlap half-squares, as illustrated in Figure~\ref{fig_n=567}A.

However, when $n=6$ and $n=7$, the $1\times n$ hole is not sufficiently long to distinguish two long edge segments on each side and this argument is insufficient. In each of these cases, we identify two long edge segments and extend the width of their counting region by one half-square in each direction along the edge of the hole. This allows us to guarantee that at least 3 half-squares in the counting region are associated with the hole, because we can count the half-squares marked with an asterisk in Figure \ref{fig_longedge}B if need be. Figures \ref{fig_n=567}B and \ref{fig_n=567}C illustrate this procedure for the $n=6$ and $n=7$ cases.

This method of counting associated overlap half-squares can be extended to $1\times n$ holes of arbitrary dimension, as illustrated in Figure \ref{fig_n=large}A. \\

\begin{figure}[h!]
    \centering
    
    \begin{subfigure}{0.5\textwidth}
        \centering
        \begin{tikzpicture}[scale=0.6]
            \halfsquaregrid{9}{9}
            \hole{0}{1}{16}{1}
            \li{0}{1}{0.5}{0.5}
            \li{0}{1}{8}{9}
            \li{0.5}{0.5}{8.5}{8.5}
            \li{8}{9}{8.5}{8.5}
            \vertex{0}{1}{}
            \vertex{1}{1}{}
            \vertex{1}{2}{}
            \vertex{2}{2}{}
            \vertex{2}{3}{}
            \vertex{3}{3}{}
            \vertex{3}{4}{}
            \vertex{4}{4}{}
            \vertex{4}{5}{}
            \vertex{5}{5}{}
            \vertex{5}{6}{}
            \vertex{6}{6}{}
            \vertex{6}{7}{}
            \vertex{7}{7}{}
            \vertex{7}{8}{}
            \vertex{8}{8}{}
            \vertex{8}{9}{}
            \rbracket{1}{1}{2}
            \rbracket{3}{3}{2}
            \rbracket{5}{5}{2}
            \rbracket{7}{7}{2}
            \lbracket{1}{2}{2}
            \lbracket{3}{4}{2}
            \lbracket{6}{7}{3}
            \node[blue] at (0,0) {0.5};
            \node[blue] at (9,9) {0.5};
        \end{tikzpicture}
        \caption {  }
    \end{subfigure}%
    \begin{subfigure}{0.5\textwidth}
        \centering
        \begin{tikzpicture}[scale=0.6]
            \halfsquaregrid{10}{10}
            \hole{1}{1}{16}{1}
            \li{1}{1}{9}{9}
            \li{1.5}{0.5}{9.5}{8.5}
            \vertex{1}{1}{}
            \vertex{2}{1}{}
            \vertex{2}{2}{}
            \vertex{3}{2}{}
            \vertex{3}{3}{}
            \vertex{4}{3}{}
            \vertex{4}{4}{}
            \vertex{5}{4}{}
            \vertex{5}{5}{}
            \vertex{6}{5}{}
            \vertex{6}{6}{}
            \vertex{7}{6}{}
            \vertex{7}{7}{}
            \vertex{8}{7}{}
            \vertex{8}{8}{}
            \vertex{9}{8}{}
            \vertex{9}{9}{}
            \rbracket{2}{1}{2}
            \rbracket{4}{3}{2}
            \rbracket{6}{5}{2}
            \rbracket{8}{7}{2}
            \lbracket{2}{2}{2}
            \lbracket{4}{4}{2}
            \lbracket{6}{6}{2}
            \lbracket{8}{8}{2}
            \filldraw (0.5, 0) circle (2pt);
            \filldraw (0.75, 0.25) circle (2pt);
            \filldraw (1, 0.5) circle (2pt);
            \filldraw (9.5, 9) circle (2pt);
            \filldraw (9.75, 9.25) circle (2pt);
            \filldraw (10, 9.5) circle (2pt);
        \end{tikzpicture}
        \caption{  }
    \end{subfigure}%
    
    \caption{ Associations for the holes of dimensions $1 \times 16$ and $1 \times \infty$. }
    \label{fig_n=large}
\end{figure}

\noindent\textbf{Claim 3. }Any length $n$ segment of the $1\times\infty$ hole is associated with at least $n$ overlap half-squares. \\

This result is analogous to Claim 2, and can be verified by counting distinct long edge segments along the edge of an arbitrary section of the $1\times\infty$ hole. Refer to Figure \ref{fig_n=large}B for a visual example.

Lemma \ref{lemma_moreoverlapthanholes} now follows from Claims 1-3.
\end{proof}

We are now equipped to establish our second main result.

\begin{theorem}
\label{thm_r=2}
If $t>2$, then $\delta_{t,2}(G_{\infty})=\frac{1}{2(t-1)^2}$.
\end{theorem}

\begin{proof} Assume for contradiction that there exists a pattern $\mathbb{T}$ of towers corresponding to a $(t,2)$ broadcast on the infinite grid with broadcast density less than $\frac{1}{2(t-1)^2}$. By Lemma \ref{lemma_moreholesthanoverlap}, this broadcast has more hole half-squares than overlap half-squares. However, by Lemma \ref{lemma_moreoverlapthanholes}, this broadcast has at least as many overlap half-squares as hole half-squares. This is a contradiction, and thus $\mathbb{T}$ is impossible.
\end{proof}

\section{$(t,r)$ Broadcast Isomorphisms on the Infinite Grid}

Blessing et al. prove that the optimal $(3,3)$ and $(2,1)$ broadcast densities are equal on large grids, and offer a conjecture equivalent to the following. 

\begin{conjecture}{(Blessing et al., \cite{blessing2015t})}
For all $t,r \in \mathbb{Z}^+$, the optimal $(t,r)$ and $(t+1,r+2)$ broadcast densities are equal on $G_\infty$.
\label{conj_blessing}
\end{conjecture}

In this section, we present a series of counterexamples to Conjecture \ref{conj_blessing}. However, we believe that the conjecture is true in the $t>2$, $r=1$ case. Before proceeding to the proof of Theorem \ref{thm_conj}, we present a trivial counterexample to Conjecture \ref{conj_blessing}.

\begin{proposition}\label{prop:bound}
$\delta_{1,1}(G_\infty) > \delta_{2,3}(G_\infty)$.
\label{prop:11_counterexample}
\end{proposition}
\begin{proof}
By Theorem \ref{thm_r=1}, the optimal $(1,1)$ broadcast $\mathbb{T}^*_{1,1}$ is the set of all vertices of the infinite grid. Thus $\delta_{1,1}(G_\infty) = 1.$ 

\begin{figure}[h!]
    \begin{center}
        \begin{tikzpicture}[scale=0.7]
            \clip (-0.9, -0.9) rectangle (10.9, 3.9);
            \grid[lightgray]{10}{3}
            \tower{-1}{-1}{1}{}
            \tower{2}{-1}{1}{}
            \tower{3}{-1}{1}{}
            \tower{6}{-1}{1}{}
            \tower{7}{-1}{1}{}
            \tower{10}{-1}{1}{}
            \tower{11}{-1}{1}{}
            \tower{0}{0}{1}{}
            \tower{1}{0}{1}{}
            \tower{4}{0}{1}{}
            \tower{5}{0}{1}{}
            \tower{8}{0}{1}{}
            \tower{9}{0}{1}{}
            
            \tower{-1}{1}{1}{}
            \tower{2}{1}{1}{}
            \tower{3}{1}{1}{}
            \tower{6}{1}{1}{}
            \tower{7}{1}{1}{}
            \tower{10}{1}{1}{}
            \tower{11}{1}{1}{}
            \tower{0}{2}{1}{}
            \tower{1}{2}{1}{}
            \tower{4}{2}{1}{}
            \tower{5}{2}{1}{}
            \tower{8}{2}{1}{}
            \tower{9}{2}{1}{}

            \tower{-1}{3}{1}{}
            \tower{2}{3}{1}{}
            \tower{3}{3}{1}{}
            \tower{6}{3}{1}{}
            \tower{7}{3}{1}{}
            \tower{10}{3}{1}{}
            \tower{11}{3}{1}{}
            \tower{0}{4}{1}{}
            \tower{1}{4}{1}{}
            \tower{4}{4}{1}{}
            \tower{5}{4}{1}{}
            \tower{8}{4}{1}{}
            \tower{9}{4}{1}{}
        \end{tikzpicture}
        \caption{ An optimal $(2,3)$ broadcast.}
        \label{(2,3)-optimal}
    \end{center}
\end{figure}
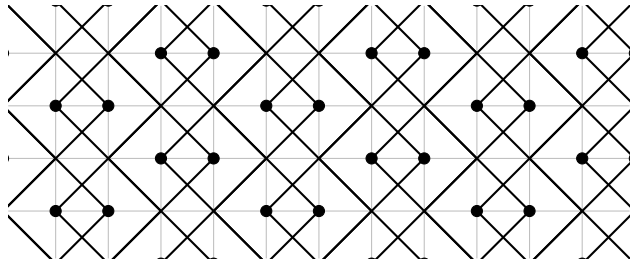

Figure \ref{(2,3)-optimal} displays the broadcast outlines of the set of towers 
\[\mathbb{T}^*_{2,3} = \{T_{i,j} \ : \ i + 2j \ mod \ 4 \in \{0,1\}\}.\]
We show that $\mathbb{T}^*_{2,3}$ is an optimal $(2,3)$ broadcast with density $\frac{1}{2}$, which suffices to establish Proposition \ref{prop:11_counterexample}.

In $\mathbb{T}^*_{2,3}$ with $t=2$, each vertex $v \in \mathbb{T}^*_{2,3}$ receives 2 signal from itself and 1 signal from an adjacent tower vertex, and each vertex $v \in V_\infty$, $v \not\in\mathbb{T}^*_{2,3}$ receives 1 signal each from three adjacent tower vertices. Thus $\mathbb{T}^*_{2,3}$ is a $(2,3)$ broadcast. Because half of the vertices on each horizontal line are contained in $\mathbb{T}^*_{2,3}$, $\mathbb{T}^*_{2,3}$ has density $\frac{1}{2}$. Finally, because a tower with signal strength $t=2$ emits a total of $6$ signal and every vertex $v \in V_\infty$ requires $3$ signal in a $(2,3)$ broadcast, the density of a $(2,3)$ broadcast is at least $\frac{3}{6} = \frac{1}{2}$ and thus $\mathbb{T}^*_{2,3}$ is optimal.
\end{proof}

In light of Proposition~\ref{prop:bound}, one may expect that other counterexamples to the conjecture exist for $t>2$ and $r=1$. However, our computer implementations \cite{DHRprogram} do not provide any counterexamples for  $t>2$ and $r=1$. Instead, they provide ample evidence for the tightness of the upper bound presented below.
\begin{theorem}
If $t>2$, then $\delta_{t,3}(G_\infty) \leq \delta_{t-1,1}(G_\infty)$.
\label{thm_conj}
\end{theorem}

\begin{proof}
The result follows from proving that $\mathbb{T}^*_{t,1}=\mathbb{T}(2t^2-2t+1,2t-1)$ is a $(t+1,3)$ broadcast. To see this, consider a vertex $v \in G_\infty$. If $sig(v,T) \geq 2$ in $\mathbb{T}^*_{t_0,1}$ with $t=t_0$, then $sig(v,T) \geq 3$ in $\mathbb{T}^*_{t_0,1}$ with $t=t_0+1$. Alternatively, if $sig(v,T_0) = 1$ in $\mathbb{T}^*_{t_0,1}$ with $t=t_0$, then $v$ is adjacent to at least one vertex $u$ with $sig(u,T_0) = 0$. As $\mathbb{T}^*_{t,1}$ is a $(t,1)$-broadcast, $sig(u,T_1) \geq 1$ for some tower vertex $T_1$. Thus $sig(v,T_0) = 2$ and $sig(v,T_1) \geq 1$ in $\mathbb{T}^*_{t_0,1}$ with $t=t_0+1$.
\end{proof}

\begin{figure}[h!]
    \centering
    \begin{tikzpicture}[scale=0.7]
        \grid{20}{8}
        \clip (-0.9,-0.9) rectangle (20.9, 8.9);
        \tower{4}{-3}{4}{}
        \tower{11}{-2}{4}{}
        \tower{18}{-1}{4}{}
        \tower{0}{0}{4}{}
        \tower{7}{1}{4}{}
        \tower{14}{2}{4}{}
        \tower{21}{3}{4}{}
        \tower{3}{4}{4}{}
        \tower{10}{5}{4}{}
        \tower{17}{6}{4}{}
        \tower{-1}{7}{4}{}
        \tower{6}{8}{4}{}
        \tower{13}{9}{4}{}
        \tower{20}{10}{4}{}
        \tower{2}{11}{4}{}
    \end{tikzpicture}
    \caption{ $\mathbb{T}^*_{(4,1)}$, or $\mathbb{T}(25,7)$, is a $(5,3)$ broadcast. }
    \label{fig:windmill_t=4}
\end{figure}
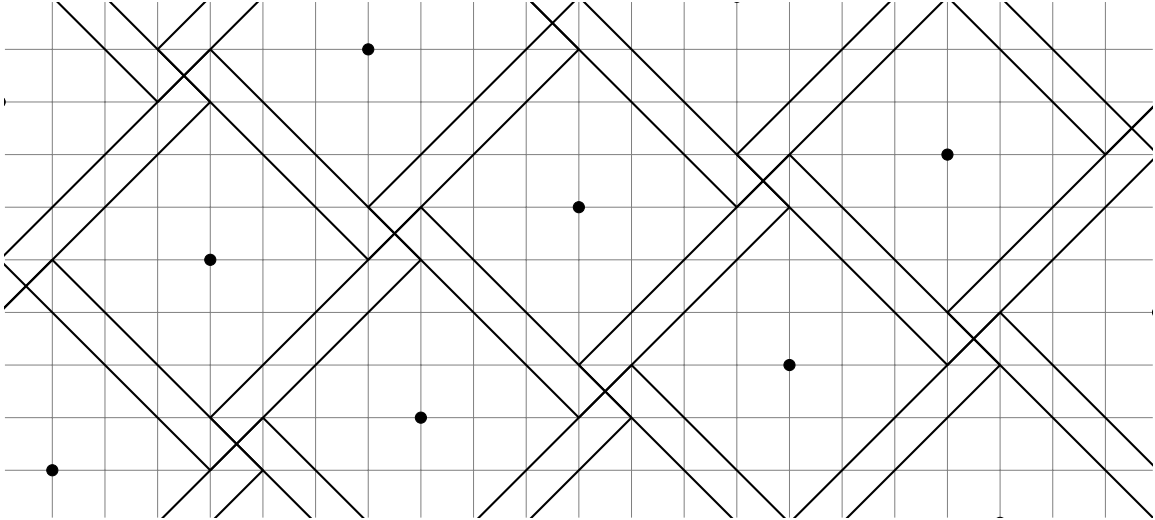
Figure \ref{fig:windmill_t=4} displays the standard broadcast $\mathbb{T}^*_{4,1}$ with $t=5$.
At this point, we have established optimal $(t,r)$ broadcast densities for all $t \in \mathbb{Z}^+$ when $r=1,2,$ and provided an upper bound on the optimal $(t,r)$ broadcast density for all $t \in \mathbb{Z}^+$ when $r=3$. We believe that the upper bound stated in Theorem \ref{thm_conj} is tight. However, a direct proof of this statement is not accessible to the methods previously employed in this paper.

\begin{table}[h!]
 \begin{centering}
 \begin{tabular}
 {||c c c c c||} 
 \hline
 t & \parbox[t]{3cm}{Optimal $(t,1)$\\ Broadcast}  & \parbox[t]{3.75cm}{Best Standard\\$(t+1,3)$ Broadcast}  & \parbox[t]{3.75cm}{Best Standard\\$(t+2,5)$ Broadcast}  & \parbox[t]{3.75cm}{Best Standard \\$(t+3,7)$ Broadcast}\\ 
 \hline\hline
 2 & $\mathbb{T}$(5,3)   & {$\mathbb{T}$(5,3)}   & $\mathbb{T}$(8,2)   & $\mathbb{T}$(11,2)\\ 
 \hline
 3 & $\mathbb{T}$(13, 5) & {$\mathbb{T}$(13, 5)} & $\mathbb{T}$(14,4)  & $\mathbb{T}$(19,7)\\
 \hline
 4 & $\mathbb{T}$(25,7)  & {$\mathbb{T}$(25,7)}  & $\mathbb{T}$(26,10) & $\mathbb{T}$(29,12)\\
 \hline
 5 & $\mathbb{T}$(41,9)  & {$\mathbb{T}$(41,9)}  & $\mathbb{T}$(42,16) & $\mathbb{T}$(43,12)\\
 \hline
 6 & $\mathbb{T}$(61,11) & {$\mathbb{T}$(61,11)} & $\mathbb{T}$(62,26) & $\mathbb{T}$(65,18)\\
  \hline
\end{tabular}
\end{centering}
\caption{ Optimal $(t,1)$ broadcasts and the best standard $(t+1,3)$, $(t+2,5)$, and $(t+3,7)$ broadcasts. }
\label{conjecturefalsetable}
\end{table}

Although Conjecture \ref{conj_blessing} appears to be true in the $t=1$, $r>1$ case, it is false in general. To demonstrate this, we present a table of standard $(t,r)$ broadcasts computed using the code presented in \cite{DHRprogram}.
Table \ref{conjecturefalsetable} compares optimal $(t, 1)$ broadcasts computed using Theorem \ref{thm_r=1} to the standard $(t+1,3)$, $(t+2,5)$, and $(t+3,7)$ broadcasts of lowest density. Because the density of the standard $(t,r)$ broadcast $\mathbb{T}(d,e)$ is $\frac{1}{d}$, the broadcasts in columns 3 and 4 with lower densities than the corresponding optimal $(t, 1)$ broadcasts provide additional counterexamples to Conjecture \ref{conj_blessing}. Many more standard $(t,r)$ broadcasts, including examples that demonstrate $\delta_{t,2}(G_\infty)\neq\delta_{t+1,4}(G_\infty)$, are provided in \cite{DHRprogram}.

Although Conjecture \ref{conj_blessing} is false in many cases, the best standard $(t+1,3)$ broadcasts provide further evidence that it may be true in the $t>1$, $r=1$ case. These results motivate the following revised conjecture.

\newtheorem*{conjecture:t=3}{Conjecture \ref{conj_t=3}}
\begin{conjecture:t=3}
For all $t>2$, $\delta_{t,3}(G_\infty) = \delta_{t-1,1}(G_\infty)$.
\end{conjecture:t=3}

\section{Further Questions}

\noindent In this section we state several open questions concerning $(t,r)$ broadcasts on the infinite~grid.

\begin{enumerate}
\item (Conjecture \ref{conj_t=3}). What is the optimal $(t,{3})$ broadcast density on the infinite grid? 
\item What is the optimal $(t,r)$ broadcast density on the infinite grid for $r>3$?
\item The proof of Proposition \ref{prop:11_counterexample} provides an example of an optimal non-standard $(2,3)$ broadcast. Moreover, no standard $(2,3)$ broadcast is optimal. Are there other pairs of positive integers $t$ and $r$ for which no standard $(t,r)$ broadcast is optimal? Because the standard $(t,r)$ broadcast of lowest density can be easily computed \cite{DHRprogram}, answering this question in the negative would provide many optimal $(t,r)$ broadcasts.
\item Given a standard broadcast $\mathbb{T}(d,e)$, what is the minimum $t$ required to give each grid vertex a signal $r$?
\item Given a standard broadcast $\mathbb{T}(d,e)$ and a signal $t$, what is the minimum signal $r$ over all vertices?
\end{enumerate}

\bibliography{main}

\end{document}